\documentclass[12pt, reqno]{amsart}

\author[M.~Balcerzak]{Marek Balcerzak}
\address{Institute of Mathematics, Lodz University of Technology, al. Politechniki 8, 93-590 Lodz, Poland}
\email{marek.balcerzak@p.lodz.pl}
\author[P.~Leonetti]{Paolo Leonetti}
\address{Department of Economics, Universit\`a degli Studi dell'Insubria, via Monte Generoso 71, 21100 Varese, Italy}
\email{leonetti.paolo@gmail.com}

\keywords{Baire property; meager sets; translates of sets; topological zero-one laws; characteristic functions; Banach spaces.}

\subjclass[2020]{Primary: 54H05. Secondary: 26A21, 28A05, 22A05, 46B20.}

\title{Linear relations modulo meager sets}

\usepackage[T1]{fontenc}
\usepackage{amsmath}
\usepackage{amssymb}
\usepackage{amsthm}
\usepackage[left=3.5cm, right=3.5cm]{geometry}
\usepackage{hyperref}
\usepackage{fancyhdr}
\usepackage{enumitem}
\usepackage{comment}
\usepackage{nicefrac}
\usepackage{comment}
\usepackage{bm}
\usepackage{mathrsfs}
\usepackage{graphicx}
\usepackage[utf8]{inputenc}
\usepackage{cancel}
\usepackage{mathtools}

\usepackage{tikz}
\usetikzlibrary{decorations.pathreplacing,matrix,arrows,positioning,automata,shapes,shadows,calc,fadings,decorations,snakes,through,intersections}

\AtBeginDocument{%
   \def\MR#1{}
}

\newtheorem{thm}{Theorem}[section]
\newtheorem{cor}[thm]{Corollary}
\newtheorem{lem}[thm]{Lemma}
\newtheorem{prop}[thm]{Proposition}

\theoremstyle{definition} 
\newtheorem{defi}[thm]{Definition}
\let\olddefi\defi
\renewcommand{\defi}{\olddefi\normalfont}
\newtheorem{question}{Question}
\let\oldquestion\question
\renewcommand{\question}{\oldquestion\normalfont}
\newtheorem{example}[thm]{Example}
\let\oldexample\example
\renewcommand{\example}{\oldexample\normalfont}
\newtheorem{rmk}[thm]{Remark}
\let\oldrmk\rmk
\renewcommand{\rmk}{\oldrmk\normalfont}

\hypersetup{
    pdftitle={Linear relations modulo meager sets},
    pdfauthor={Balcerzak and Leonetti},
    pdfmenubar=false,
    pdffitwindow=true,
    pdfstartview=FitH,
    colorlinks=true,
    linkcolor=blue,
    citecolor=green,
    urlcolor=cyan
}

\providecommand{\MR}[1]{}

\providecommand{\MR}{\relax\ifhmode\unskip\space\fi MR }

\begin{document}

\maketitle
\thispagestyle{empty}

\begin{abstract}
We prove several topological zero-one laws. First, we show that, if a subset $A$ of a Banach space $X$ has the Baire property, then $A$ admits a somewhere dense set of vectors $x \in X$ such that $A+x$ agrees with $A$ modulo meager sets if and only if it is either meager or comeager. Additional equivalent conditions are given if $X=\mathbb{R}$. 

Second, we prove that if $A_1,\ldots,A_k\subseteq \mathbb{R}$ are subsets with the Baire property, $\alpha_1,\ldots,\alpha_k$ are nonzero reals with distinct finite sums, and $(x_{n,i}: n\in \omega)$ are injective real sequences which converge to $0$ for each $i=1,\ldots,k$, then  
$$
\sum_{i=1}^k \alpha_i(\bm{1}_{A_i+x_{n,i}}-\bm{1}_{A_i})=0
$$
modulo meager sets for all $n \in \omega$ if and only if each $A_i$ is either meager or comeager. 
Finally, we provide some additional results in the case where the $\alpha_1,\ldots,\alpha_k$ do not have distinct finite sums. For instance, if $k=2$ and $\alpha_1=\alpha_2$, then the above equality holds modulo meager sets for all $n \in \omega$ if and only if each $A_i$ is either meager or comeager or if $\{A_1,A_2\}$ is a partition of $\mathbb{R}$ modulo meager sets. Additional characterizations are given in the case $k\ge 3$. These results yield the category analogues of several results by Fejzi\'{c}, Freiling, and Rinne in [J. London Math. Soc.~\textbf{82} (2010), no. 3, 717--732]. 
\end{abstract}


\section{Introduction and Main results}\label{sec:intro}

It is known that a measurable set $A\subseteq \mathbb{R}$ has arbitrarily small periods (that is, $A+x_n=A$ for all $n\in \omega$, where $(x_n)$ stands for a positive real sequence with limit $0$) if and only if either $A$ or its complement has measure $0$, see Exercise 3 in Rudin's book \cite[p.156]{MR0924157}. (Note that $A+x_n=A$ can be rewritten equivalently as $\bm{1}_{A+x_n}-\bm{1}_A=0$, where $\bm{1}_S$ stands for the characteristic function of a set $S$.) 

Motivated by this fact, Fast, Fejzi\'{c}, Freiling, and Rinne studied in \cite{MR2083818, MR2458270, MR2739064} the analogous situation of measurable sets $A_1,\ldots,A_k\subseteq \mathbb{R}$ which satisfy  
\begin{equation}\label{eq:measurablecase}
\forall n \in \omega, \quad \sum_{i=1}^k \alpha_i(\bm{1}_{A_i+x_{n,i}}-\bm{1}_{A_i})=0 \,\,\,\text{ a.e.}, 
\end{equation}
(that is, modulo sets of measure zero), where $\alpha_1,\ldots,\alpha_k$ are given nonzero reals, and $(x_{n,i}: n \in \omega)$ is an injective real sequence with limit $0$ for each $i=1,\ldots,k$. 

In this manuscript, we provide several category analogues of the above results, thereby answering an open question in \cite{MR2458270}. We divide our work into two main parts.
\begin{enumerate}[label={\rm (\roman{*})}]
\item In the first part, given a Banach space $X$, we provide necessary and sufficient conditions on a set $A\subseteq X$ with the Baire property to admit a somewhere dense set of vectors $x \in X$ for which $A+x=A$ modulo meager sets, see Theorem \ref{thm:maincategorycaseone}. It turns out that this happens if and only if either $A$ or its complement is meager, see Theorem \ref{thm:maincategorycaseone}. Additional equivalent conditions in the one-dimensional case $X=\mathbb{R}$ can be found in Corollary \ref{cor:onedimensional}. These results can be seen as topological zero-one laws, in the same spirit of \cite[Theorem 8.46]{K}, cf. also \cite{MR969916, MR507257, MR345084}. 

\item In the second part, we specialize to the case $X=\mathbb{R}$, and characterize the category analogue of \eqref{eq:measurablecase} for sets $A_1,\ldots,A_k$ with the Baire property. More precisely, we prove that if all finite sums of the $\alpha_j$ are distinct, then each set $A_i$ has to be either meager or comeager, see Theorem \ref{thm:distinctFS}. In addition, if $k=2$ and $\alpha_1=\alpha_2$, then at least one of the following two conditions holds: (a) each $A_i$ is either meager or comeager, or (b) $A_1=\mathbb{R}\setminus A_2$ modulo meager sets, see Theorem \ref{thm:twodimensional}. Additional results are given in the general case $k\ge 3$, see Theorem \ref{thm:generalcase} and Proposition \ref{prop:equivalentE1E5}. 
\end{enumerate}

We present our main results in Section \ref{sec:mainresults}. Their proofs are given in Section \ref{sec:proofs}. We remark that our techniques are completely different from the measure counterparts given in \cite{MR2083818, MR2458270, MR2739064}. 

\section{Main Results}\label{sec:mainresults}

Given subsets $A,B$ of a topological group $X$ (written additively, with neutral element $0$), 
we write 
$
A+B:=\{x+y: x \in A, y \in B\},
$ 
and we shorten $A+x:=A+\{x\}$ for each $x \in X$. Let also $A^\prime$ be the set of accumulation points of $A$, and $\partial A$ the boundary points of $A$. Recall that $A$ is nowhere dense if its closure has empty interior; in the opposite, $A$ is said to be somewhere dense. In addition, $A$ is meager if it is a countable union of nowhere dense sets. The family of meager subsets of $X$ is denoted by $\mathscr{M}$. We also write $A\sim B$ if the symmetric difference $A\bigtriangleup B$ is meager. For each $A\subseteq X$, define
$$
\mathcal{D}(A):=\left\{x \in X: A+x \sim A\,\right\}.
$$ 
In other words, $\mathcal{D}(A)$ is the 
symmetric group generated by the equivalence classes of elements of $A$ in the quotient space $X/\sim$, 
cf. e.g. \cite{Reiher2024}.

\subsection{One set in Banach spaces.} 
Our first main result is a higher dimensional category analogue of \cite[Theorem 1]{MR2083818}, cf. also \cite[p. 156]{MR2739064}; in addition, it answers an open question in \cite{MR2458270}.
\begin{thm}\label{thm:maincategorycaseone}
    Let $X$ be a Banach space and fix a subset $A\subseteq X$ with the Baire property. Then the following are equivalent\textup{:}
    \begin{enumerate}
    [label=\textup{(}\textsc{A}\arabic*\textup{)}]%
    \item \label{item:01firstthm} $A\sim \emptyset$ or $A\sim X$\textup{;}

    \item \label{item:02firstthm} $\mathcal{D}(A)=X$\textup{;}

\item \label{item:03firstthm} $\mathcal{D}(A)$ is somewhere dense\textup{.} 
    \end{enumerate}
Hence, either $\mathcal{D}(A)=X$ or $\mathcal{D}(A)$ is nowhere dense. 
If, in addition, $X$ is one-dimensional, then each of the above conditions is also equivalent to each of the next ones\textup{:}
\begin{enumerate}
    [label=\textup{(}\textsc{A}\arabic*\textup{)}]%
\setcounter{enumi}{3}
    \item \label{item:04firstthm} $0 \in \mathcal{D}(A)^\prime$\textup{;}
    \item \label{item:05firstthm} $\mathcal{D}(A)^\prime \neq \emptyset$\textup{;}
    \item \label{item:06firstthm} $\mathcal{D}(A)\neq \{nx: n \in \mathbb{Z}\}$ for all $x \in X$\textup{.}
\end{enumerate}
\end{thm}

Some remarks are in order, cf. also Remark \ref{rmk:noncomplete} below. 

\begin{rmk}\label{rmk:notalltopgroups}
    The equivalences in the first part of Theorem \ref{thm:maincategorycaseone} do not hold for all topological groups. Indeed, consider the Polish group $X=2^\omega$ and define the open set $A:=\{x \in 2^\omega: x_0=0\}$. Then neither $A$ nor $A^c$ are meager, and $\mathcal{D}(A)=A$. Therefore items \ref{item:01firstthm} and \ref{item:02firstthm} fail, while item \ref{item:03firstthm} holds.  
    \end{rmk}

    \begin{rmk}\label{rmk:notdimensiongreaterone}
    The equivalences in the second part of Theorem \ref{thm:maincategorycaseone} do not hold in dimension greater than $1$. Indeed, consider the Euclidean space $X=\mathbb{R}^d$ with $d\ge 2$ and define the open set $A:=\mathbb{R}^{d-1}\times (0,1)$. Then $\mathcal{D}(A)=\mathbb{R}^{d-1}\times \{0\}$. It follows that items \ref{item:01firstthm}--\ref{item:03firstthm} fail, while items \ref{item:04firstthm}--\ref{item:06firstthm} hold. 
    \end{rmk}

Since in \cite{MR2083818, MR2458270, MR2739064} the results are formulated only in the real case $X=\mathbb{R}$ and in terms of injective real sequences $(s_n: n \in \omega)$ such that 
\begin{equation}\label{eq:AsimAplusxnforalln}
    \forall n \in \omega, \quad A+s_n\sim A,
\end{equation}
we formulate the equivalences of Theorem \ref{thm:maincategorycaseone} analogously.
\begin{cor}\label{cor:onedimensional}
    Let $A\subseteq \mathbb{R}$ be a set with the Baire property. 
    Then the following are equivalent\textup{:}
    \begin{enumerate}
    [label=\textup{(}\textsc{B}\arabic*\textup{)}]

     \item \label{item:01firstthmcor} $A\sim \emptyset$ or $A\sim \mathbb{R}$\textup{;}
    
     \item \label{item:02firstthmcor} 
     there exists an injective sequence $(s_n)$ satisfying \eqref{eq:AsimAplusxnforalln} such that $\lim_n s_n=0$\textup{;}

     \item \label{item:03firstthmcor} there exists an injective bounded sequence $(s_n)$ satisfying \eqref{eq:AsimAplusxnforalln}\textup{;}

    \item \label{item:04firstthmcor} 
    there exists an injective sequence $(s_n)$ satisfying \eqref{eq:AsimAplusxnforalln} such that 
     $$
     \inf_{n,k \in \omega, n\neq k}|s_n-s_k|=0\textup{;}
     $$


 \item \label{item:06firstthmcor} 
    there exists an injective sequence $(s_n)$ satisfying \eqref{eq:AsimAplusxnforalln} and there is no $r>0$ such that $s_n/r$ is an integer for all $n \in \omega$\textup{;}

     \item \label{item:07firstthmcor} 
    there exists an injective sequence $(s_n)$ satisfying \eqref{eq:AsimAplusxnforalln} 
     such that $\{s_n: n \in \omega\}$ is dense\textup{.}
    \end{enumerate}

In particular, either $\mathcal{D}(A)=\mathbb{R}$ or $0$ is an isolated point of $\mathcal{D}(A)$.    
\end{cor}


%
%
%

In the case where $\mathcal{D}(A)$ contains two nonzero elements with irrational ratio, we have the following consequence: 
\begin{cor}\label{cor:Qlinearlyindependent}
Let $A\subseteq \mathbb{R}$ be a set with the Baire property, and suppose that there exist $x,y \in \mathbb{R}$ such that 
$$
A\sim A+x\sim A+y
$$
and 
$\{x,y\}$ is linearly independent over $\mathbb{Q}$. Then $A\sim \emptyset$ or $A\sim \mathbb{R}$.
\end{cor}


\subsection{Many sets in the real space}

Given a subset $A$ of a topological group $X$, we write $\bm{1}_A$ for the characteristic function of $A$ (as in $X=\mathbb{R}$). Moreover, given a statement $P(z)$, with $z \in X$, which can be either true or false, we write $P \bmod{\mathscr{M}}$ if $\{z \in X: P(z) \text{ is false}\} \in \mathscr{M}$. In particular, we can write equivalently $A=B \bmod{\mathscr{M}}$ for $A\sim B$, and also 
$$
\mathcal{D}(A)=\{x \in X: \bm{1}_{A+x}- \bm{1}_A=0 \bmod{\mathscr{M}}\}.
$$  
With these premises, given subsets $A_1,\ldots,A_k\subseteq X$ and nonzero reals $\alpha_1,\ldots,\alpha_k$, define
$$
\mathcal{D}_{\alpha_1,\ldots,\alpha_k}(A_1,\ldots,A_k):=\left\{x \in X^k: \sum_{i=1}^k\alpha_i(\bm{1}_{A_i+x_i}-\bm{1}_{A_i})=0\,\,\bmod{\mathscr{M}}\right\}.
$$
Moreover, if $\alpha_1=\cdots=\alpha_k$, we shorten $\mathcal{D}(A_1,\ldots,A_k):=\mathcal{D}_{\alpha_1,\ldots,\alpha_k}(A_1,\ldots,A_k)$. 

\begin{rmk}\label{rmk:atleastonemeagercomeager}
Let us suppose that at least one set in $\{A_1,\ldots,A_k\}$ is meager or comeager, let us say $A_k$. Then 
$$
\mathcal{D}_{\alpha_1,\ldots,\alpha_k}(A_1,\ldots,A_k)=
\mathcal{D}_{\alpha_1,\ldots,\alpha_{k-1}}(A_1,\ldots,A_{k-1})\times X.
$$
In particular, $\mathcal{D}_{\alpha_1,\ldots,\alpha_k}(A_1,\ldots,A_k)=X^k$ if each $A_i$ is either meager or comeager. Hence, to simplify our statements, it is sufficient to consider the case where each set $A_i$ is neither meager nor comeager. 
\end{rmk}


\medskip

Hereafter, we specialize to the one-dimensional case $X=\mathbb{R}$. We say that $\{A_1,\ldots,A_k\}$ is a partition of $\mathbb{R}$ modulo $\mathscr{M}$ if $A_1 \cup \cdots \cup A_k$ is comeager and $A_i \cap A_j$ is meager for all distinct $i,j \in \{1,\ldots,k\}$. Moreover, we write $S^c:=\mathbb{R}\setminus S$ for each $S\subseteq \mathbb{R}$. For notational convenience, given nonzero reals $\alpha_1,\ldots,\alpha_k$, define 
$$
\mathrm{FS}(\alpha_1,\ldots,\alpha_k):=\left\{\,\sum_{i \in I}\alpha_i: I \subseteq \{1,\ldots,k\}\right\}.
$$
Note that a variant of the map $\mathrm{FS}$ has been used in Hindman's finite sum theorem, see \cite[Theorem 3.1]{MR349574}; cf. also \cite{MR4356195, 
FKL2026, 
MR1887003}.
Lastly, given a positive integer $k$, let 
$$
\Delta_k:=\{(x,\ldots,x) \in \mathbb{R}^k: x \in \mathbb{R}\}
$$
be the principal diagonal of $\mathbb{R}^k$. 

In the following result, we study the case where $k$ is an arbitrary positive integer and all finite sums of $\{\alpha_1,\ldots,\alpha_k\}$ are distinct (or, equivalently, $\mathrm{FS}(\alpha_1,\ldots,\alpha_k)$ has maximal cardinality). 
\begin{thm}\label{thm:distinctFS}
Let $A_1,\ldots,A_k\subseteq \mathbb{R}$ be sets with the Baire property and fix nonzero reals $\alpha_1,\ldots,\alpha_k$ such that all finite sums are distinct. Then the following are equivalent\textup{:}
\begin{enumerate}
[label=\textup{(}\textsc{C}\arabic*\textup{)}]
\item \label{item:1FS} each $A_i$ is either meager or comeager\textup{;} 

\item \label{item:2FS} $\mathcal{D}_{\alpha_1,\ldots,\alpha_k}(A_1,\ldots,A_k)=\mathbb{R}^k$\textup{;}

\item \label{item:3FS}$\Delta_k\subseteq \mathcal{D}_{\alpha_1,\ldots,\alpha_k}(A_1,\ldots,A_k)$\textup{;}

\item \label{item:4FS} there exists an injective sequence $(x_n)$ with values in 
$\mathcal{D}_{\alpha_1,\ldots,\alpha_k}(A_1,\ldots,A_k)$ such that 
$\lim_n x_n=(0,\ldots,0)$ and $x_{n,i}\neq 0$ for all $n \in \omega$ and all $i \in \{1,\ldots,k\}$\textup{.}
\end{enumerate}
Hence, either $\mathcal{D}_{\alpha_1,\ldots,\alpha_k}(A_1,\ldots,A_k)$ is equal to $\mathbb{R}^k$, or there exists a neighborhood $U$ of $(0,\ldots,0)$ such that every point of 
$
\mathcal{D}_{\alpha_1,\ldots,\alpha_k}(A_1,\ldots,A_k)\cap U
$ 
has at least one zero coordinate.
\end{thm}

\medskip

Next, we present a characterization in the case $k=2$ and $\alpha_1=\alpha_2$, which provides (a refinement of) the category analogue of \cite[Theorem 5]{MR2083818}. Together with Theorem \ref{thm:distinctFS} and Remark \ref{rmk:atleastonemeagercomeager}, this completes the study of the case $k=2$. 
(It is worth noting that, differently from our previous results, this is not a type of topological zero-one law.)
\begin{thm}\label{thm:twodimensional}
    Let $A,B\subseteq \mathbb{R}$ be sets with the Baire property such that each one is neither meager nor comeager. Then the following are equivalent\textup{:}
    \begin{enumerate}
    [label=\textup{(}\textsc{D}\arabic*\textup{)}]
    \item \label{item1twodimensional} $\{A,B\}$ is a partition of $\mathbb{R}$ modulo $\mathscr{M}$\textup{;}

    \item \label{item2twodimensional} 
    $\mathcal{D}(A,B)=\{(x,x+rz): x \in \mathbb{R}, z \in \mathbb{Z}\}$ for some $r \in \mathbb{R}$\textup{;}

    \item \label{item2Btwodimensional} $\Delta_2\subseteq \mathcal{D}(A,B)$\textup{;}

    \item \label{item3twodimensional} $(0,0)\in \mathcal{D}(A,B)^\prime$, i.e., there exists an injective sequence $(x_n,y_n)$ with values in $\mathbb{R}^2$ which converges to $(0,0)$ and such that 
    $$
    \forall n \in \omega, \quad 
\bm{1}_{A+x_{n}}+\bm{1}_{B+y_n}=\bm{1}_A+\bm{1}_B\bmod{\mathscr{M}}. 
    $$
    \end{enumerate}
\end{thm}

Observe that item \ref{item2twodimensional} above can be rewritten equivalently as $\mathcal{D}(A,B)=\Delta_2+(\{0\}\times r\mathbb{Z})$ for some $r \in \mathbb{R}$. 
The next two examples show the case where $k\ge 3$ and $\mathrm{FS}(\alpha_1,\ldots,\alpha_k)$ does not have maximal cardinality (for instance, if $\alpha_1=\cdots=\alpha_k=1$) can be substantially different. 
\begin{example}\label{example:3A}
    Suppose that $X=\mathbb{R}$, and set $R_1:=\{0\}$, $R_2:=\{1,2\}$, and $R_3:=\{0,1\}$. Thus, define the sets
    \begin{equation}\label{eq:dfinizioneexample3case}
    A_i:=\left\{x \in \mathbb{R}: \lfloor x \rfloor \in R_i+3\mathbb{Z}\right\}
    \end{equation}
    for each $i \in \{1,2,3\}$. Then it is straightforward to check that 
    $$
    \mathcal{D}(A_1,A_2,A_3)=(3\mathbb{Z})^3\cup \left((3\mathbb{Z}+1)\times (3\mathbb{Z}-1)\times 3\mathbb{Z}\right).
    $$
    In particular, $A_1\cup A_2\cup A_3 =\mathbb{R}$, $A_1\cap A_2\cap A_3=\emptyset$, and $(0,0,0)\notin \mathcal{D}(A_1,A_2,A_3)^\prime$. 
\end{example}

\begin{example}\label{example:3B}
    Suppose that $X=\mathbb{R}$, and set $R_1:=\{0,1\}$, $R_2:=\{0,2\}$, and $R_3:=\{1,2\}$. Define the sets $A_1,A_2,A_3$ as in \eqref{eq:dfinizioneexample3case}. Then 
    $$
    \mathcal{D}(A_1,A_2,A_3)=\Delta_3+\, \{0\}\times \left((3\mathbb{Z})^2)\cup ((3\mathbb{Z}-1)\times (3\mathbb{Z}+1))\right).
    $$
    In particular, $A_1\cup A_2\cup A_3 =\mathbb{R}$, $A_1\cap A_2\cap A_3=\emptyset$, and $(0,0,0)\in \mathcal{D}(A_1,A_2,A_3)^\prime$. Note that, although $\Delta_3\subseteq \mathcal{D}(A_1,A_2,A_3)$, it is not true that $\{A_1,A_2,A_3\}$ is a partition of $\mathbb{R}$ modulo $\mathscr{M}$. However, 
    there exists a partition $\{B_1,B_2,B_3\}$ of $\mathbb{R}$ modulo $\mathscr{M}$ such that $A_1=B_1\cup B_2$, $A_2=B_2\cup B_3$, and $A_3=B_3\cup B_1$ modulo $\mathscr{M}$. 
\end{example}


In our last main result, we study the case of arbitrary $k\ge 2$ and $\alpha_1=\cdots=\alpha_k$, under an additional restriction on the sets $A_i$. 

\begin{thm}\label{thm:generalcase}
Given an integer $k\ge 2$, let $A_1,\ldots,A_k\subseteq \mathbb{R}$ be sets with the Baire property such that each $A_i$ is neither meager nor comeager. 
In addition, assume that the boundary points of each $A_i$ are isolated, that is, 
\begin{equation}\label{eq:uglyconstraint}
\forall i \in \{1,\ldots,k\}, \quad \left(\partial A_i\right)^\prime = \emptyset. 
\end{equation}
Then the following are equivalent\textup{:}
\begin{enumerate}
[label=\textup{(}\textsc{E}\arabic*\textup{)}]
\item \label{item:1general} there exists an integer $h \in \{1,\ldots,k-1\}$ such that 
$$
\left\{\bigcap_{i \in I} A_i: I\subseteq \{1,\ldots,k\} \text{ and }|I|=h\right\}
$$
is a partition of $\mathbb{R}$ modulo $\mathscr{M}$\textup{;}

\item \label{item:2general} there exists $h \in \{1,2,\ldots,k-1\}$ such that $\sum_{i=1}^k \bm{1}_{A_i}=h \bmod{\mathscr{M}}$\textup{;} 

\item \label{item:3general} $\mathcal{D}(A_1,\ldots,A_k)=\Delta_k+(\{0\}\times \mathcal{D}(A_2,\ldots,A_k))$\textup{;}

\item \label{item:4general} for every neighborhood $U$ of $(0,\ldots,0)$ there exists 
$x=(x_1,\ldots,x_k)\in \mathcal{D}(A_1,\ldots,A_k)\cap U$ such that $x_i\neq 0$ for all $i\in \{1,\ldots,k\}$\textup{;}

\item \label{item:5general}$\Delta_k\subseteq \mathcal{D}(A_1,\ldots,A_k)$\textup{;}

\item \label{item:6general} there exists an injective sequence
$(x_n)$ with values in $\mathcal{D}(A_1,\ldots,A_k)$ such that
$x_n\to(0,\ldots,0)$ and $x_{n,i}\ne0$ for every
$i\in\{1,\ldots,k\}$ and every $n\in \omega$\textup{.}
\end{enumerate}
Hence, either $\Delta_k\subseteq \mathcal{D}(A_1,\ldots,A_k)$, or there exists a neighborhood $U$ of $(0,\ldots,0)$ such that every point of 
$
\mathcal{D}(A_1,\ldots,A_k)\cap U
$ 
has at least one zero coordinate.
\end{thm}

In particular, in the case of $k=3$ sets, the above characterization implies that the existence of arbitrarily small vectors $x=(x_1,x_2,x_3)\in \mathcal{D}(A_1,A_2,A_3)$ with $x_1x_2x_3\neq 0$ is equivalent to exactly one of the following conditions: either $\{A_1,A_2,A_3\}$ is a partition of $\mathbb{R}$ modulo $\mathscr{M}$ (choosing $h=1$ in item \ref{item:1general}) or $\{A_1\cap A_2, A_2\cap A_3, A_1\cap A_3\}$ is a partition of $\mathbb{R}$ modulo $\mathscr{M}$ (choosing $h=2$ in item \ref{item:1general}). Notice that the latter situation is precisely what happens in Example \ref{example:3B}. In addition, differently from the case $k=2$, we do not have a representation of $\mathcal{D}(A_1,\ldots,A_k)$ which is \textquotedblleft more explicit\textquotedblright\, than $\Delta_k+\mathcal{D}(A_2,\ldots,A_k)$ in item \ref{item:3general}, provided that $k\ge 3$ (which is also evident in Example \ref{example:3B}).

\begin{rmk}\label{rmk:nonemptyaccboundary}
Since the sets $\mathcal{D}(\cdot,\ldots,\cdot)$ are invariant modulo meager modifications, the hypothesis \eqref{eq:uglyconstraint} can be replaced by the following weaker requirement: 
$$
\forall i \in \{1,\ldots,k\}, \exists B_i\subseteq \mathbb{R}, \quad 
A_i\sim B_i\,\, \text{ and }\,\,(\partial B_i)^\prime=\emptyset.
$$
The latter is satisfied by additional examples. For instance, if $C\subseteq [0,1]$ is the ternary Cantor set, then $A:=C^c$ satisfies $(\partial A)^\prime=\partial A=C$, hence all boundary points of $A$ are not isolated. On the other hand, $A\sim \mathbb{R}$ and, of course, $\partial \mathbb{R}=\emptyset$. 

On the other hand, it is quite possible that a set $A\subseteq \mathbb{R}$ with the Baire property satisfies $(\partial B)^\prime\neq \emptyset$ for all $B\subseteq \mathbb{R}$ with $A\sim B$. For instance, the set 
\begin{equation}\label{eq:Astar}
A_\star:=\bigcup_{n\ge 1}\left(\left(\frac{1}{2^{2n+1}}, \frac{1}{2^{2n}}\right)\cup \left(1-\frac{1}{2^{2n}}, 1-\frac{1}{2^{2n+1}}\right)\right).
\end{equation}
satisfies $\{0,1\}\subseteq (\partial B)^\prime$ for all $B\subseteq \mathbb{R}$ such that $A_\star\sim B$. 
\end{rmk}

The formulation of items \ref{item:4general} and \ref{item:6general} cannot be weakened to the ordinary condition 
$(0,\ldots,0)\in \mathcal{D}(A_1,\ldots,A_k)^\prime$. Indeed, let
$$
A_1:=(-\infty,0),\qquad A_2:=[0,\infty),\qquad A_3:=(0,1)+2\mathbb{Z}.
$$
Then each $A_i$ has the Baire property, is neither meager nor comeager, and satisfies $(\partial A_i)^\prime=\emptyset$. Moreover, for every positive sequence $(\varepsilon_n)$ converging to $0$, the points
$$
x_n:=(\varepsilon_n,\varepsilon_n,0)
$$
belong to $\mathcal{D}(A_1,A_2,A_3)$. Hence $(0,0,0)\in \mathcal{D}(A_1,A_2,A_3)^\prime$ in the ordinary sense. However,
$$
\sum_{i=1}^3 \bm{1}_{A_i}=1+\bm{1}_{A_3}
$$
is not constant modulo $\mathscr{M}$, so item \ref{item:2general} fails. Thus the requirement that all coordinates are nonzero is essential.

In the next result, we show that items \ref{item:1general}, \ref{item:2general}, \ref{item:3general}, and \ref{item:5general} are actually equivalent, independently of Condition \eqref{eq:uglyconstraint}.
\begin{prop}\label{prop:equivalentE1E5}
Given an integer $k\ge 2$, let $A_1,\ldots,A_k\subseteq \mathbb{R}$ be sets with the Baire property such that each $A_i$ is neither meager nor comeager. Then items \ref{item:1general}, \ref{item:2general}, \ref{item:3general}, and \ref{item:5general} are equivalent.
\end{prop}

It is also worth noting that the stronger local conclusion 
$
\mathcal{D}(A_1,\ldots,A_k)\cap U=\Delta_k\cap U
$ 
near the origin is false in general, even under \eqref{eq:uglyconstraint}. For instance, let
$$
A_1:=(-\infty,0),\qquad A_2:=[0,\infty),\qquad 
A_3:=(-\infty,1),\qquad A_4:=[1,\infty).
$$
Then $\sum_{i=1}^4\bm{1}_{A_i}=2$ modulo $\mathscr{M}$, and therefore items \ref{item:1general}, \ref{item:2general}, \ref{item:3general}, and \ref{item:5general} hold. However, for every sufficiently small nonzero $t$ we have
$$
(t,t,0,0)\in \mathcal{D}(A_1,A_2,A_3,A_4)\setminus \Delta_4.
$$
Thus no neighborhood $U$ of $(0,0,0,0)$ satisfies 
$\mathcal{D}(A_1,A_2,A_3,A_4)\cap U=\Delta_4\cap U$. 

We leave it to the interested reader to check, 
in the same spirit of Theorem \ref{thm:maincategorycaseone}, whether the analogues of Theorem \ref{thm:twodimensional}, Theorem \ref{thm:generalcase}, and Proposition \ref{prop:equivalentE1E5} hold in an infinite-dimensional setting.

\section{Preliminaries}\label{sec:preliminaries}

\subsection{Topology and subgroups} 
Let us prove in this section some intermediate results, starting with the topological and the additive structure of 
$\mathcal{D}_{\alpha_1,\ldots,\alpha_k}(A_1,\ldots,A_k)$. 

\begin{lem}\label{lem:DAsubgroup}
Let $A_1,\ldots,A_k$ be subsets of a 
topological 
group $X$. Fix also nonzero reals $\alpha_1,\ldots,\alpha_k$. 
Then 
\begin{equation}\label{eq:subgroupX}
\left\{x \in X: \sum_{i=1}^k \alpha_i (\bm{1}_{A_i+x}-\bm{1}_{A_i})=0 \,\,\,\bmod{\mathscr{M}}\right\},
\end{equation}
namely, $\mathcal{D}_{\alpha_1,\ldots,\alpha_k}(A_1,\ldots,A_k) \cap \Delta_k$, is a subgroup of $X$. In particular, $\mathcal{D}(A)$ is a subgroup of $X$ for each $A\subseteq X$. 
\end{lem}
\begin{proof}
    Let $\mathcal{S}$ be the set defined in \eqref{eq:subgroupX}. Of course, we have $0 \in \mathcal{S}$. Fix $x \in \mathcal{S}$ and note that
    $$
    \forall y \in X, \quad \sum_{i=1}^k \alpha_i (\bm{1}_{A_i+x+y}-\bm{1}_{A_i+y})=0 \,\,\,\bmod{\mathscr{M}}.
    $$
    On the one hand, setting $y=-x$ we obtain that $-x \in \mathcal{S}$. On the other hand, if $y \in \mathcal{S}$, then 
    $$
    \sum_{i=1}^k \alpha_i (\bm{1}_{A_i+x+y}-\bm{1}_{A_i})
    =\sum_{i=1}^k \alpha_i (\bm{1}_{A_i+x+y}-\bm{1}_{A_i+y})+\sum_{i=1}^k \alpha_i (\bm{1}_{A_i+y}-\bm{1}_{A_i}),
    $$
    which is $0$ modulo $\mathscr{M}$. Therefore $x+y \in \mathcal{S}$. 
\end{proof}

Following \cite[p. 53]{MR1039321}, a topological space $X$ is said to be sequential if, for each subset $A\subseteq X$, $A$ is closed if and only if it is sequentially closed. 
\begin{lem}\label{lem:DAclosed}
Let $A_1,\ldots,A_k$ be subsets of a 
topological 
group $X$ which is a sequential space, and assume that they have the Baire property. Fix also nonzero reals $\alpha_1,\ldots,\alpha_k$. Then 
$$
\left\{(x_1,\ldots,x_k) \in X^k: \sum_{i=1}^k \alpha_i\,\bm{1}_{A_i+x_i}=0\,\,\,\,\bmod{\mathscr{M}}\right\}
$$
is closed. Consequently, $\mathcal{D}_{\alpha_1,\ldots,\alpha_k}(A_1,\ldots,A_k)$ is closed. 
\end{lem}
\begin{proof}
Suppose without loss of generality that each $A_i$ is open. Fix a convergent sequence $(x^n: n \in \omega)$ with values in $X^k$ such that
\begin{equation}\label{eq:forallnmeagrr}
\forall n \in \omega, \quad \sum_{i=1}^k \alpha_i\,\bm{1}_{A_i+x^n_i}=0\,\,\,\,\bmod{\mathscr{M}},
\end{equation}
where $x^n=(x^n_1,\ldots,x^n_k)$ for each $n \in \omega$. 
Denote by $(\eta_1,\ldots,\eta_k)$ its limit (recalling that topological groups are Hausdorff) and let $S_i$ be the set of boundary points of each $A_i+\eta_i$. Note that each $S_i$ is nowhere dense, hence $S:=\bigcup_i S_i \in \mathscr{M}$. In addition, for each $z \in X\setminus S$, there exists $n_{z} \in \omega$ such that 
$\bm{1}_{A_i+\eta_i}(z)=\bm{1}_{A_i+x^n_i}(z)$ for all $i\in \{1,\ldots, k\}$ and all integers $n\ge n_z$. Hence
\begin{displaymath}
    \begin{split}
        \left\{z\in X: \sum_{i=1}^k \alpha_i\,\bm{1}_{A_i+\eta_i}(z)\neq 0\right\}
        &\subseteq S\cup \left\{z\in X\setminus S: \sum_{i=1}^k \alpha_i\,\bm{1}_{A_i+\eta_i}(z)\neq 0\right\}\\
        &= S\cup \left\{z\in X\setminus S: \sum_{i=1}^k \alpha_i\,\bm{1}_{A_i+x^{n_z}_i}(z)\neq 0\right\}\\
        &\subseteq S\cup \bigcup_{n \in \omega}\left\{z\in X: \sum_{i=1}^k \alpha_i\,\bm{1}_{A_i+x^{n}_i}(z)\neq 0\right\}.
    \end{split}
\end{displaymath}
It follows by \eqref{eq:forallnmeagrr} and the fact that $\mathscr{M}$ is a $\sigma$-ideal that $\sum_{i=1}^k \alpha_i\,\bm{1}_{A_i+\eta_i}=0\,\,\,\,\bmod{\mathscr{M}}$. Therefore the required set is sequentially closed, which completes the proof. 

The second part follows since 
$$
\mathcal{D}_{\alpha_1,\ldots,\alpha_k}(A_1,\ldots,A_k)=\iota^{-1}[\mathcal{A}\cap \mathcal{B}],
$$
where 
$\iota: X^k\to X^{2k}$ is the continuous map given by $x\mapsto (x_1,\ldots,x_k,0,\ldots,0)$, 
$$
\mathcal{A}:=
\left\{(x_1,\ldots,x_{2k}) \in X^{2k}: \sum_{i=1}^k \alpha_i(\bm{1}_{A_i+x_i}-\bm{1}_{A_i+x_{k+i}})=0 \bmod{\mathscr{M}}\right\}
$$
is closed by the first part of the proof, 
and 
$$
\mathcal{B}:=\left\{(x_1,\ldots,x_{2k}) \in X^{2k}: x_{k+1}=x_{k+2}=\cdots=x_{2k}=0\right\},
$$
is closed as well.  
\end{proof}

We will use also the following well-known fact, see e.g. \cite[Exercise 21, p. 28]{MR1669668}. 
\begin{lem}\label{lem:closedadditivesubgroup}
    Let $Y$ be a closed additive subgroup of $\mathbb{R}$. Then either $Y=r\mathbb{Z}$ for some $r \in \mathbb{R}$ or $Y=\mathbb{R}$. 
\end{lem}


\subsection{Empty intersections Vs small shifts} We continue with some preliminary results on the structure of the intersections $\bigcap_i (A_i+x_i)$, provided that $\bigcap_i A_i=\emptyset$. Hereafter, given a normed vector space $X$, we write $B(x,r):=\{y \in X: \|x-y\|<r\}$ for the open ball with center $x \in X$ and radius $r \ge 0$. 

\begin{lem}\label{lem:intersection}
Let $X$ be a normed vector space, and fix nonempty subsets $A_1,\ldots,A_k\subseteq X$ with $A_1\cap \cdots \cap A_k=\emptyset$. Fix also vectors $x_1,\ldots,x_k \in X$. Then
$$
(A_1+x_1)\cap \cdots \cap (A_k+x_k) \subseteq \bigcup_{i=1}^k \bigcup_{z \in \partial A_i}B(z,r), 
$$
where $r:=2\max\{\|x_1\|,\ldots,\|x_k\|\}$. 
\end{lem}
\begin{proof}
    First, note that the claim holds if $x_1=\cdots=x_k=0$ or if $\bigcap_j (A_j+x_j)=\emptyset$. Hence, let us suppose hereafter that $r>0$. Fix a vector $x \in \bigcap_j (A_j+x_j)$. Since $\bigcap_j A_j=\emptyset$, there exists $i \in \{1,\ldots,k\}$ such that $x\notin A_i$. Let $Z$ be the set of all vectors $z$ in the closure of $A_i$ such that 
    \begin{equation}\label{eq:minimaldistancesets}
    \forall a \in A_i, \quad \|x-z\|\le \|x-a\|+\frac{r}{4}.
    \end{equation}
    Note that $Z$ is nonempty. In addition, $Z\cap \partial A_i\neq \emptyset$. To prove it, pick a vector $z \in Z$ and suppose that it is an interior point of $A_i$. Then there exists a maximal $\varepsilon_0 \in (0,1)$ such that $z_\varepsilon:=z+\varepsilon (x-z)$ belongs to $Z\cap A_i$ for all $\varepsilon \in (0,\varepsilon_0)$: in fact, $\|x-z_\varepsilon\|=(1-\varepsilon)\|x-z\|$. By the continuity of the norm, we have $z_{\varepsilon_0} \in Z \cap \partial A_i$. 
    To conclude, since $x \in A_i+x_i$, there exists $a_i \in A_i$ such that $x=a_i+x_i$. By the definition of $z_{\varepsilon_0}$, it follows that 
    $$
    \|x-z_{\varepsilon_0}\|\le \|x-a_i\|+\frac{r}{4}=\|x_i\|+\frac{r}{4}<r,
    $$
    therefore $x \in B(z_{\varepsilon_0},r)$. 
\end{proof}

\begin{rmk}
    It is worth recalling that the positive error term $r/4$ in Inequality \eqref{eq:minimaldistancesets} is necessary. Indeed, it is known that in every infinite-dimensional Banach space $X$ there exists a nonempty closed set $C\subseteq X$ and a vector $x \in X\setminus C$ such that $x$ has no nearest point in $C$, see e.g. \cite[Theorem 4.1]{MR1012624}. 
\end{rmk}

In the special case of sets $A_1,\ldots,A_k\subseteq \mathbb{R}$ with isolated boundary points, we have the following additional property. 
\begin{lem}\label{lem:mainideaintersections}
Fix open sets $A_1,\ldots,A_k\subseteq \mathbb{R}$ which satisfy \eqref{eq:uglyconstraint}, and suppose that $A_1\cap \cdots \cap A_k=\emptyset$ and that each set $A_i^c$ has no isolated point. Then, for all $z \in (\partial A_1)\cup \cdots \cup (\partial A_k)$, there exists $\varepsilon>0$ such that
$$
z\notin \partial \left(\bigcap_{i=1}^k(A_i+x_i)\right) 
$$
for all nonzero $x_1,\ldots,x_k \in (-\varepsilon,\varepsilon)$. 
In other words, either $z$ is an interior point of $\bigcap_{i} (A_i+x_i)$ or it is an interior point of its complement. 
\end{lem}
\begin{proof}
Fix $z \in T:=(\partial A_1)\cup \cdots \cup (\partial A_k)$. Thanks to \eqref{eq:uglyconstraint}, there exists $\varepsilon>0$ such that $|z-t|\ge 4\varepsilon$ for all $t \in T\setminus \{z\}$. Decreasing $\varepsilon$ if necessary, we may also assume that, for each $i\in \{1,\ldots,k\}$, exactly one of the following alternatives holds:
\begin{enumerate}[label={\rm (\roman{*})}]
\item $(z-4\varepsilon,z+4\varepsilon)\subseteq A_i$;
\item $(z-4\varepsilon,z+4\varepsilon)\subseteq A_i^c$;
\item $z\in \partial A_i$, $(z,z+4\varepsilon)\subseteq A_i$, and $(z-4\varepsilon,z)\subseteq A_i^c$;
\item $z\in \partial A_i$, $(z-4\varepsilon,z)\subseteq A_i$, and $(z,z+4\varepsilon)\subseteq A_i^c$.
\end{enumerate}
Indeed, if $z\notin \partial A_i$, then $z$ is an interior point of either $A_i$ or $A_i^c$. If $z\in \partial A_i$, then, since there are no other boundary points of $A_i$ in $(z-4\varepsilon,z+4\varepsilon)$, each side of $z$ is contained either in $A_i$ or in $A_i^c$. The two sides cannot both be contained in $A_i$, because then $z$ would be an isolated point of $A_i^c$, and they cannot both be contained in $A_i^c$, because then $z$ would not be a boundary point of $A_i$.

Now suppose that $x_1,\ldots,x_k$ are nonzero reals in $(-\varepsilon,\varepsilon)$. Define
\begin{displaymath}
\begin{split}
I^0&:=\{i\in \{1,\ldots,k\}:(z-4\varepsilon,z+4\varepsilon)\subseteq A_i^c\},\\
I^+&:=\{i\in \{1,\ldots,k\}: z\in \partial A_i \text{ and } (z,z+4\varepsilon)\subseteq A_i\},\\
I^-&:=\{i\in \{1,\ldots,k\}: z\in \partial A_i \text{ and } (z-4\varepsilon,z)\subseteq A_i\}.
\end{split}
\end{displaymath}
If $I^0\neq \emptyset$, then
$$
B(z,2\varepsilon)\cap \bigcap_{i=1}^k(A_i+x_i)=\emptyset,
$$
and therefore $z$ is an interior point of the complement of $\bigcap_i(A_i+x_i)$.

Assume now that $I^0=\emptyset$. The indices $i$ for which $(z-4\varepsilon,z+4\varepsilon)\subseteq A_i$ do not impose any restriction inside $B(z,2\varepsilon)$. Hence
\begin{displaymath}
\begin{split}
B(z,2\varepsilon)\cap \bigcap_{i=1}^k(A_i+x_i)
&=
B(z,2\varepsilon)\cap \bigcap_{i\in I^+}(z+x_i,z+4\varepsilon+x_i)
       \cap \bigcap_{i\in I^-}(z-4\varepsilon+x_i,z+x_i)\\
&=
(z+a,z+b),
\end{split}
\end{displaymath}
where
$$
a:=\max\left(\{-2\varepsilon\}\cup \{x_i:i\in I^+\}\right)
\quad\text{and}\quad
b:=\min\left(\{2\varepsilon\}\cup \{x_i:i\in I^-\}\right),
$$
with the convention that $(z+a,z+b)=\emptyset$ if $a\ge b$. Since each $x_i$ is nonzero, neither $a$ nor $b$ can be equal to $0$ unless it is one of the artificial endpoints $-2\varepsilon$ or $2\varepsilon$. In all cases, $z$ is either an interior point of $(z+a,z+b)$ or an interior point of its complement. Thus
$$
z\notin \partial\left(\bigcap_{i=1}^k(A_i+x_i)\right).
$$
\end{proof}

\subsection{Main technical tools} 
We proceed now to the main lemma which will allow to prove Theorem \ref{thm:distinctFS} and Theorem \ref{thm:twodimensional}. 
\begin{lem}\label{lem:twodimensionalcase}
Fix subsets $A_1,\ldots,A_k\subseteq \mathbb{R}$ 
such that each $A_i$ has the Baire property and is neither meager nor comeager. 
Suppose that there exists a sequence $(x_n)$ with values in $\mathbb{R}^k$ such that 
each $x_{n,i}$ is nonzero, and $\lim_n x_n=(0,\ldots,0)$, and 
$$
\forall n \in \omega, \quad 
(A_1+x_{n,1})\cap \cdots \cap (A_k+x_{n,k}) \sim A_1\cap \cdots \cap A_k.
$$
Then $A_1\cap \cdots \cap A_k$ is meager. 
\end{lem}
\begin{proof}
Without loss of generality, each $A_i$ is open. Indeed, since each $A_i$ has the
Baire property, there exists an open set $O_i$ such that $A_i \sim O_i$, and the
hypothesis and the conclusion are invariant under replacing each $A_i$ by $O_i$.

Set
\[
        U := A_1 \cap \cdots \cap A_k
\]
and, for each $n \in \omega$,
\[
        U_n := (A_1+x_{n,1}) \cap \cdots \cap (A_k+x_{n,k}).
\]
Then $U$ and all the sets $U_n$ are open. By hypothesis, $U_n \sim U$ for all
$n \in \omega$. Hence, since $U_n \triangle U$ is open and meager, it is empty.
Thus $U_n=U$ for all $n \in \omega$. 

Suppose, for the sake of contradiction, that $U$ is not meager. Since $U$ is
open, it is nonempty. Moreover, $U \neq \mathbb R$, because otherwise
$A_i=\mathbb R$ for every $i \in \{1,\ldots,k\}$, contradicting the assumption
that the sets $A_i$ are not comeager. Hence $U$ is a nonempty proper open
subset of $\mathbb R$.

Let $(a,b)$ be a connected component of $U$. At least one of $a,b$ is finite.
We treat the case $a \in \mathbb R$; the case $b \in \mathbb R$ follows in the
same way, or equivalently by applying the argument to the reflected sets
$-A_1,\ldots,-A_k$.

For each $i \in \{1,\ldots,k\}$, let $(a_i,b_i)$ be the connected component of
$A_i$ which contains $(a,b)$. Then $a_i \leq a$ and $b_i \geq b$. Since
$(a,b)$ is a connected component of $U=\bigcap_{i=1}^k A_i$, we have
\[
        \max_{1 \leq i \leq k} a_i = a.
\]
Indeed, if this maximum were strictly smaller than $a$, then all sets $A_i$
would contain a nonempty interval immediately to the left of $a$, and hence
$U$ would strictly extend the component $(a,b)$ to the left.

Put
\[
        I := \{i \in \{1,\ldots,k\}: a_i=a\}.
\]
Then $I$ is nonempty. Fix a point $q \in (a,b)$. Choose $\varepsilon>0$ so small
that, whenever $|t_i|<\varepsilon$ for all $i$, the following two conditions hold:
\[
        q \in (a_i+t_i,b_i+t_i) \quad \text{for all } i \in \{1,\ldots,k\},
\]
and
\[
        \max_{1 \leq i \leq k}(a_i+t_i)
        =
        a+\max_{i \in I} t_i .
\]
This is possible because $q \in (a,b) \subseteq (a_i,b_i)$ for all $i$, and
$a_i<a$ for all $i \notin I$.

Now fix $n$ large enough so that $0<|x_{n,i}|<\varepsilon$ for all
$i \in \{1,\ldots,k\}$. Since $q \in (a_i+x_{n,i},b_i+x_{n,i})$ for all $i$, the
connected component of $U_n$ which contains $q$ is
\[
        \bigcap_{i=1}^k (a_i+x_{n,i},b_i+x_{n,i})
        =
        \left(
        \max_{1 \leq i \leq k}(a_i+x_{n,i}),
        \min_{1 \leq i \leq k}(b_i+x_{n,i})
        \right).
\]
By the choice of $\varepsilon$, the left endpoint of this component is
\[
        a+\max_{i \in I} x_{n,i}.
\]
Since all the numbers $x_{n,i}$ are nonzero, also
$\max_{i \in I}x_{n,i} \neq 0$. Therefore this left endpoint is different from
$a$.

On the other hand, since $U_n=U$ for all $n$, the connected component
of $U_n$ containing $q$ must be precisely the connected component $(a,b)$ of
$U$ containing $q$. In particular, its left endpoint must be $a$, a contradiction.

Therefore $U=A_1 \cap \cdots \cap A_k$ is meager, as required.
\end{proof}

\begin{rmk}\label{rmk:lemmaonedimension}
More precisely, the proof of the Lemma \ref{lem:twodimensionalcase} shows the following local
fact. Suppose that all sets $A_i$ are open, set $U:=\bigcap_i A_i$, and let
$(a,b)$ be a connected component of $U$ with $a\in\mathbb R$. For each
$i\in\{1,\ldots,k\}$, let $(a_i,b_i)$ be the connected component of $A_i$ which
contains $(a,b)$. Then $\max_i a_i=a$. Hence, for every $q\in(a,b)$ and all
sufficiently small $x=(x_1,\ldots,x_k)\in\mathbb R^k$, the connected component
of
$
        \bigcap_i(A_i+x_i)
$ 
which contains $q$ has left endpoint
\[
        a+\max\{x_i:a_i=a\}.
\]
In particular, if $x_n\to(0,\ldots,0)$ and $x_{n,i}\neq0$ for all $i,n$, then,
for all sufficiently large $n$, the set $\bigcap_i(A_i+x_{n,i})$ contains a
nonempty open interval of the form
\[
        \left(a+\max\{x_{n,i}:a_i=a\},q\right),
\]
provided that $a+\max\{x_{n,i}:a_i=a\}<q$; also, the leftmost point of this
interval is the left endpoint of the connected component of
$\bigcap_i(A_i+x_{n,i})$ containing $q$.
\end{rmk}

In the next result we prove an extension of Lemma \ref{lem:twodimensionalcase} under an additional restriction on the sets $A_i$, which will be needed for the proof of Theorem \ref{thm:generalcase}. 
\begin{prop}\label{prop:generaldimensionalcasenew}
Fix 
$A_1,\ldots,A_p\subseteq \mathbb{R}$ such that each $A_i$ has the Baire property, it is neither meager nor comeager, and $(\partial A_i)^\prime=\emptyset$. 
Let $\{P_1,\ldots,P_q\}$ be a partition of $\{1,\ldots,p\}$ into nonempty sets and suppose that there exists a sequence $(x_n: n \in \omega)$ with values in $\mathbb{R}^p$ such that  
each $x_{n,i}$ is nonzero, $\lim_n x_n=(0,\ldots,0)$, and 
\begin{equation}\label{eq:claimedequality}
\forall n \in \omega, \quad \bigcup_{j=1}^q \bigcap_{i \in P_j} (A_i+x_{n,i}) \sim \bigcup_{j=1}^q \bigcap_{i \in P_j} A_i.
\end{equation}

Then $\bigcup_{j=1}^q \bigcap_{i \in P_j} A_i$ is either meager or comeager. 
\end{prop}
\begin{proof}
%
Note that \(\partial A_i\) is nowhere dense, hence meager. Define $O_i:=\operatorname{Int}(\overline{A_i})$ for each $i \in \{1,\ldots,p\}$. 
Then \(A_i\triangle O_i\subseteq \partial A_i\), and also 
$ 
\partial O_i\subseteq \partial A_i. 
$ 
Thus \(A_i\sim O_i\), \(O_i\) is regular open, and \((\partial O_i)^\prime=\emptyset\). Since the statement is invariant under replacing each \(A_i\) by a meager equivalent set, we may assume without loss of generality that each \(A_i\) is regular open and satisfies \((\partial A_i)^\prime=\emptyset\). 
For notational convenience, define 
$$
U_j:=\bigcap_{i \in P_j} A_i
\quad \text{ and }\quad 
U:=\bigcup_{j=1}^q U_j
$$
for all $j\in \{1,\ldots,q\}$. Similarly, define
$$
V_{n,j}:=\bigcap_{i \in P_j} (A_i+x_{n,i})
\quad \text{ and }\quad 
V_n:=\bigcup_{j=1}^q V_{n,j}
$$
for all  $n \in \omega$ and $j\in \{1,\ldots,q\}$. Lastly, set \(F:=\bigcup_{i=1}^p \partial A_i\). Then all points of \(F\) are isolated, and 
$
\partial U_j\cup \partial U \subseteq F
$ 
for every \(j\). 
The standing hypothesis \eqref{eq:claimedequality} can be rewritten as $V_n \sim U$ for all $n \in \omega$. 
Suppose hereafter for the sake of contradiction that $U$ is neither meager nor comeager. 

Notice that all the sets $A_i$ and $U$ are (nonempty finite or countably infinite) unions of open intervals.  Notice that all sets $A_i$ and $U$ cannot be dense (indeed, in the opposite case, they would be open dense sets, hence comeager). 
In particular, the complement $U^c$ has nonempty interior. Let $(x,y)$ be a nonempty open interval contained in $U^c$, and define
\begin{equation}\label{eq:definitiona}
a:=\sup\{z \in \mathbb{R}\cup \{\infty\}: (x,z)\cap U \in \mathscr{M}\}. 
\end{equation}

In other words, $(x,a)$ is an open interval contained in $U^c$, with $a$ maximal modulo meager sets. Suppose that $a \in \mathbb{R}$ (the case \(a=\infty\) is handled symmetrically, by reflecting the real line and applying the same argument to the reflected sets). 
Since the points of $F$ are isolated, there exists a sufficiently small $\varepsilon>0$ such that $B(a,2\varepsilon) \cap F=\{a\}$, so that 
$(a-\varepsilon,a) \cap U =\emptyset$ and $(a,a+\varepsilon)\subseteq U$.

At this point, for each open set $G\subseteq \mathbb{R}$ define 
$$\alpha(G):=\inf (G\cap (a-\varepsilon,\infty))$$ 
where, as usual, we assume the convention that $\alpha(G)=+\infty$ if $\sup(G)\le a-\varepsilon$. 
Note that the maximality of $a$ implies that 
\begin{equation}\label{eq:maximala}
a=\alpha(U)=
\min_{j \in \{1,\ldots,q\}} \, \alpha(U_j), 
\end{equation}
and $\alpha(U_j)\ge \max_{i \in P_j} \, \alpha(A_i)$ for all $j \in \{1,\ldots,q\}$. 
In particular, we have that $\alpha(U_j)\ge a$ for all $j \in \{1,\ldots,q\}$ by \eqref{eq:maximala} and that 
the set
$$
J:=\left\{j \in \{1,\ldots,q\}: \alpha(U_j)=a\right\}.
$$ 
is nonempty. On the other hand, if $j \in \{1,\ldots,q\}\setminus J$, then $\alpha(U_j) \ge a+2\varepsilon$. Similarly, as in \eqref{eq:maximala}, since $V_n \sim U$ for all $n\in \omega$, we have 
\begin{equation}\label{eq:maximalb}
\forall n \in \omega, \quad a=\alpha(V_n)=\min_{j \in \{1,\ldots,q\}}\alpha(V_{n,j}).
\end{equation}
Indeed, if \(\alpha(V_n)<a\), then \(V_n\setminus U\) contains a nonempty open interval contained in \((a-\varepsilon,a)\), contradicting \(V_n\sim U\). If \(\alpha(V_n)>a\), then \(U\setminus V_n\) contains a nonempty open interval contained in \((a,a+\varepsilon)\), again contradicting \(V_n\sim U\).

Lastly, fix an integer $n_0 \in \omega$ such that 
$0<|x_{n,i}|<\varepsilon/4$ for all $i \in \{1,\ldots,p\}$ and $n\ge n_0$. 
In particular, it follows by \eqref{eq:maximalb} that, for each $n\ge n_0$, there exists $j \in \{1,\ldots,q\}$ such that $\alpha(V_{n,j})=a$. 

To conclude, we show that, after increasing \(n_0\) if necessary,
\[
\forall n\ge n_0,\ \forall j\in\{1,\ldots,q\},\qquad \alpha(V_{n,j})>a.
\]

First, fix \(j\in J\), so that \(\alpha(U_j)=a\). Since \(B(a,2\varepsilon)\cap F=\{a\}\), we have
\[
(a-\varepsilon,a)\cap U_j=\emptyset
\quad\text{and}\quad
(a,a+\varepsilon)\subseteq U_j.
\]
Applying Lemma \ref{lem:twodimensionalcase}, more precisely the local conclusion recorded in Remark \ref{rmk:lemmaonedimension}, to the family \((A_i:i\in P_j)\), we may increase \(n_0\) so that, for every \(n\ge n_0\), there is a nonempty set \(I_{n,j}\subseteq P_j\) such that
\[
\left(a+\max_{i\in I_{n,j}}x_{n,i},a+\frac{\varepsilon}{2}\right)\subseteq V_{n,j},
\]
and the leftmost point \(a+\max_{i\in I_{n,j}}x_{n,i}\) is maximal modulo meager sets.

Since \(V_n\sim U\) and \((a-\varepsilon,a)\cap U=\emptyset\), we must have 
$
V_{n,j}\cap(a-\varepsilon,a)=\emptyset.
$ 
Hence 
$
a+\max_{i\in I_{n,j}}x_{n,i}\ge a.
$ 
Because all \(x_{n,i}\) are nonzero, the maximum cannot be \(0\). Therefore 
$ 
\max_{i\in I_{n,j}}x_{n,i}>0,
$ 
and the maximality assertion from Remark \ref{rmk:lemmaonedimension} gives
\[
\alpha(V_{n,j})=a+\max_{i\in I_{n,j}}x_{n,i}>a.
\]

Now fix \(j\notin J\). Then \(\alpha(U_j)>a\); indeed, by the choice of \(\varepsilon\), either \(\alpha(U_j)=+\infty\), or \(\alpha(U_j)\ge a+2\varepsilon\). Put 
$ 
H_j:=(a,\alpha(U_j)),
$ 
with the convention that \(H_j=(a,\infty)\) if \(\alpha(U_j)=+\infty\). Then
\[
H_j\cap U_j=\emptyset.
\]
Apply Lemma \ref{lem:mainideaintersections} to the family consisting of the sets 
$ 
(A_i:i\in P_j)
$ 
together with the additional open set \(H_j\). The point \(a\) is a boundary point of \(H_j\), all relevant boundary points are isolated, and the complements have no isolated points.

Thus, after increasing \(n_0\) if necessary, for every \(n\ge n_0\) we may choose a nonzero \(t_{n,j}<0\), sufficiently small, such that \(a\) is not a boundary point of
\[
V_{n,j}\cap(H_j+t_{n,j}).
\]
Since \(t_{n,j}<0\) is small, \(H_j+t_{n,j}\) contains a neighborhood of \(a\). Hence \(a\) is not a boundary point of \(V_{n,j}\).

Again, from \(V_n\sim U\) and \((a-\varepsilon,a)\cap U=\emptyset\), we get 
$
V_{n,j}\cap(a-\varepsilon,a)=\emptyset.
$ 
Therefore \(a\) cannot be an interior point of \(V_{n,j}\). Since \(a\notin\partial V_{n,j}\), it follows that \(a\) is an interior point of \(V_{n,j}^c\). Hence 
$ 
\alpha(V_{n,j})>a.
$ 

We have proved that, for all sufficiently large \(n\),
\[
\forall j\in\{1,\ldots,q\},\qquad \alpha(V_{n,j})>a.
\]
Therefore
\[
\alpha(V_n)=\min_{j\in\{1,\ldots,q\}}\alpha(V_{n,j})>a,
\]
contradicting \eqref{eq:maximalb}. This concludes the proof.
\end{proof}


\section{Proofs of main results}\label{sec:proofs}

Now, let us prove our main results. 
\begin{proof}
[Proof of Theorem \ref{thm:maincategorycaseone}]
 \ref{item:01firstthm} $\implies$ 
 \ref{item:02firstthm} $\implies$
 \ref{item:03firstthm}. They are clear. 

 \medskip
 
  \ref{item:03firstthm} $\implies$ \ref{item:01firstthm}. Suppose that $\mathcal{D}(A)$ is dense in a nonempty open set $U\subseteq X$. Hence we can fix a vector $x \in \mathcal{D}(A)\cap U$. Define $V:=U-x$, which is a neighborhood of $0$. 
  Since $\mathcal{D}(A)$ is a subgroup of $X$ by Lemma \ref{lem:DAsubgroup} and $x \in \mathcal{D}(A)$, it follows that $(\mathcal{D}(A)\cap U)-x \subseteq \mathcal{D}(A)$. Therefore $\mathcal{D}(A)$ is dense in $V$. 
 
 

  
  Pick $r_0>0$ such that 
  $B(0,r_0)\subseteq V$. Suppose also for the sake of contradiction that item \ref{item:01firstthm} fails, hence both $A\notin \mathscr{M}$ and $A^c:=X\setminus A\notin \mathscr{M}$. 
  Note that it is possible to assume without loss of generality that $A$ is open. It follows that $A^c$ is closed with nonempty interior.

Pick an interior point $x_1\in A^c$ and put
$$
 r_1:=\sup\{r>0:B(x_1,r)\subseteq A^c\}.
$$
Since $A$ is nonempty open and $x_1\in\operatorname{Int}(A^c)$, we have
$0<r_1<\infty$. Moreover $B(x_1,r_1)\subseteq A^c$. Choose
$\delta>0$ such that
$$
0<\delta<\min\left\{\frac12,\frac{r_0}{4r_1}\right\}.
$$
By the definition of $r_1$, there exists $y_\delta\in A$ such that
$$
 r_1\le \|y_\delta-x_1\|<r_1+\delta r_1.
$$
Set $\rho_\delta:=\|y_\delta-x_1\|$ and
$$
 z_\delta:=x_1+\frac{(1-\delta)r_1}{\rho_\delta}(y_\delta-x_1).
$$
Then
$$
 B(z_\delta,\delta r_1)\subseteq B(x_1,r_1)\subseteq A^c,
$$
and
$$
 \|y_\delta-z_\delta\|
 =
 \rho_\delta-(1-\delta)r_1
 <2\delta r_1<r_0.
$$
Choose $\varepsilon_\delta>0$ such that
$$
 B(y_\delta,\varepsilon_\delta)\subseteq A,
 \qquad
 0<\varepsilon_\delta<\frac{\delta r_1}{4},
$$
and also
$$
 \varepsilon_\delta<r_0-\|z_\delta-y_\delta\|.
$$
Since $\mathcal D(A)$ is dense in $B(0,r_0)$, there is
$$
 d_\delta\in\mathcal D(A)\cap B(z_\delta-y_\delta,\varepsilon_\delta).
$$
Then the nonempty open set $B(y_\delta,\varepsilon_\delta)+d_\delta$ is
contained in $A+d_\delta$. If $u=w+d_\delta$, with
$w\in B(y_\delta,\varepsilon_\delta)$, then
$$
 \|u-z_\delta\|
 \le \|w-y_\delta\|+
 \|d_\delta-(z_\delta-y_\delta)\|
 <2\varepsilon_\delta
 <\delta r_1.
$$
Thus
$$
B(y_\delta,\varepsilon_\delta)+d_\delta
\subseteq B(z_\delta,\delta r_1)
\subseteq A^c.
$$
Hence $(A+d_\delta)\setminus A$ contains a nonempty open set, contradicting
$d_\delta\in\mathcal D(A)$.

\medskip

Since one-dimensional spaces are linearly homeomorphic to the reals, it is possible to assume hereafter that $X=\mathbb{R}$. Recall that $\mathcal{D}(A)$ is a closed additive subgroup of $\mathbb{R}$ by Lemma \ref{lem:DAsubgroup} and Lemma \ref{lem:DAclosed}. 

\medskip

 \ref{item:02firstthm} $\implies$ 
  \ref{item:04firstthm} $\implies$ \ref{item:05firstthm}. They are clear. 

\medskip

\ref{item:05firstthm} $\implies$ 
  \ref{item:06firstthm}. Suppose that there exists a convergent injective sequence $(x_n)$ with values in $\mathcal{D}(A)$ and denote its limit by $\eta$. 
  It follows that $\eta \in \mathcal{D}(A)$  and $y_n:=x_n-\eta \in \mathcal{D}(A)$ for all $n \in \omega$. Since $(y_n: n\in \omega)$ is an injective sequence in $\mathcal{D}(A)$ with limit $0$, we obtain that $\mathcal{D}(A)\neq r\mathbb{Z}$ for all $r \in \mathbb{R}$. 
  
  \medskip

\ref{item:06firstthm} $\implies$ 
  \ref{item:02firstthm}. 
  This follows by Lemma \ref{lem:closedadditivesubgroup}. 
\end{proof}

\medskip

\begin{rmk}\label{rmk:noncomplete}
The proof above reveals that the first three equivalences in Theorem
\ref{thm:maincategorycaseone} hold for every normed vector space $X$ which is 
a Baire space. If $X$ is meager in itself, then $X$ and all its subsets are 
meager in $X$, and the characterization becomes trivial. Thus completeness is 
not used directly; what is used is the relevant Baire-category behavior of the 
ambient normed space.
\end{rmk}



\medskip

\begin{proof}
[Proof of Corollary \ref{cor:onedimensional}]
  \ref{item:01firstthmcor} $\implies$ 
  \ref{item:02firstthmcor} $\implies$ \ref{item:03firstthmcor}. They are clear. 
  
  \medskip

  \ref{item:03firstthmcor} $\implies$ \ref{item:04firstthmcor}. It follows from the fact that every bounded sequence has a convergent subsequence. 

  \medskip

  \ref{item:04firstthmcor} $\implies$ 
  \ref{item:06firstthmcor}. 
Suppose that there exists an injective sequence $(s_n)$ with values in 
$\mathcal{D}(A)$ such that 
$\inf\{|s_n-s_k|: n,k \in \omega, n\neq k\}=0$. We claim that there is no 
$r>0$ such that $s_n/r$ is integer for all $n\in\omega$. Indeed, otherwise 
$s_n\in r\mathbb{Z}$ for all $n\in\omega$. Since $(s_n)$ is injective, it 
follows that $s_n-s_k\in r\mathbb{Z}\setminus\{0\}$ for all distinct 
$n,k\in\omega$. Hence $|s_n-s_k|\ge r$ for all distinct $n,k\in\omega$, 
which contradicts 
$\inf\{|s_n-s_k|: n,k \in \omega, n\neq k\}=0$.

\medskip

\ref{item:06firstthmcor} $\implies$ 
  \ref{item:07firstthmcor}. Recall again that $\mathcal{D}(A)$ is a closed additive subgroup of $\mathbb{R}$ by Lemma \ref{lem:DAsubgroup} and Lemma \ref{lem:DAclosed}. Hence, thanks to Lemma \ref{lem:closedadditivesubgroup}, either $\mathcal{D}(A)=\mathbb{R}$ or $\mathcal{D}(A)=r\mathbb{Z}$ for some $r\in \mathbb{R}$.
  Since $\mathcal{D}(A)$ contains each $s_n$ for all $n \in\omega$ and there is no $r>0$ such that $s_n/r$ is an integer for all $n$, it follows that $\mathcal{D}(A)\neq r\mathbb{Z}$ for all $r \in \mathbb{R}$. Therefore $\mathcal{D}(A)=\mathbb{R}$. Choosing any injective dense sequence $(t_n)$ in $\mathbb{R}$, we have $t_n\in \mathcal{D}(A)$ for all $n$, and hence $(t_n)$ satisfies \eqref{eq:AsimAplusxnforalln} and item \ref{item:07firstthmcor}. 

\medskip

\ref{item:07firstthmcor} $\implies$ 
  \ref{item:01firstthmcor}. If $(s_n)$ is as in item \ref{item:07firstthmcor}, then $\mathcal{D}(A)$ contains the dense set $\{s_n:n\in\omega\}$. Since $\mathcal{D}(A)$ is closed by Lemma \ref{lem:DAclosed}, we have $\mathcal{D}(A)=\mathbb{R}$. The conclusion follows by Theorem \ref{thm:maincategorycaseone}. 
\end{proof}

\medskip

\begin{proof}
[Proof of Corollary \ref{cor:Qlinearlyindependent}]
Note that $x,y$ are nonzero and distinct. 
Fix $r>0$ and suppose that $a:=x/r$ and $b:=y/r$ are integers. Then $a,b\neq 0$ and $ay-bx=0$. This contradicts the hypothesis that $\{x,y\}$ is linearly independent over $\mathbb{Q}$. Hence $\mathcal{D}(A)\neq r\mathbb{Z}$ for all $r \in \mathbb{R}$. The conclusion follows by Corollary \ref{cor:onedimensional}.
\end{proof}

\medskip

\begin{proof}[Proof of Theorem \ref{thm:distinctFS}]
\ref{item:1FS} $\implies$ \ref{item:2FS}. It follows by Remark \ref{rmk:atleastonemeagercomeager}.

\medskip

\ref{item:2FS} $\implies$ \ref{item:3FS} 
$\implies$ \ref{item:4FS}. They are obvious. 

\medskip

\ref{item:4FS} $\implies$ \ref{item:1FS}. 
Let 
$$
L:=\{i \in \{1,\ldots,k\}: A_i\not\sim \emptyset \text{ and } A_i\not\sim \mathbb{R}\}.
$$
Suppose, for the sake of contradiction, that $L\neq \emptyset$. By item \ref{item:4FS}, there exists an injective sequence $(x_n)$ with values in 
$\mathcal{D}_{\alpha_1,\ldots,\alpha_k}(A_1,\ldots,A_k)$ such that 
$\lim_n x_n=(0,\ldots,0)$ and $x_{n,i}\neq 0$ for all $n \in \omega$ and all $i \in \{1,\ldots,k\}$.

Since $A_i$ is either meager or comeager for each $i\notin L$, we have
$$
\forall n \in \omega,\quad 
\sum_{i \in L}\alpha_i\left(\bm{1}_{A_i+x_{n,i}}-\bm{1}_{A_i}\right)=0
\quad \bmod{\mathscr{M}}.
$$
Moreover, all finite sums of $(\alpha_i:i\in L)$ are distinct. Fix a subset $I\subseteq L$ and define $(\varepsilon_i:i\in L)\in \{0,1\}^L$ by $\varepsilon_i=1$ if and only if $i\in I$. By the maximality of the finite sums, the last displayed identity implies that
$$
\forall n \in \omega,\quad 
\bigcap_{i \in L}(A_i^{\varepsilon_i}+x_{n,i})
=
\bigcap_{i \in L}(A_i+x_{n,i})^{\varepsilon_i}
\sim
\bigcap_{i \in L}A_i^{\varepsilon_i}.
$$
Applying Lemma \ref{lem:twodimensionalcase} to the family $(A_i^{\varepsilon_i}:i\in L)$, we obtain that 
$\bigcap_{i \in L}A_i^{\varepsilon_i}$ is meager. Since $I\subseteq L$ was arbitrary, it follows that
$$
\mathbb{R}
=
\bigcup_{I\subseteq L}\bigcap_{i \in L}A_i^{\varepsilon_i}
$$
is meager, a contradiction. Therefore $L=\emptyset$, that is, each $A_i$ is either meager or comeager.
\end{proof}

\medskip

\begin{proof}
    [Proof of Theorem \ref{thm:twodimensional}]
    \ref{item1twodimensional} $\implies$ \ref{item2twodimensional}. 
    Since neither $B$ nor $B^c$ are meager, it follows by Theorem \ref{thm:maincategorycaseone} that there exists $r \in \mathbb{R}$ such that $\mathcal{D}(B)=\{rz: z \in \mathbb{Z}\}$. 
    Now, suppose that $\{A,B\}$ is a partition of $\mathbb{R}$ modulo $\mathscr{M}$. Fix $x \in \mathbb{R}$ and note that 
\begin{displaymath}
    \begin{split}
\bm{1}_{A+x}+\bm{1}_{B+x}=\bm{1}_{\mathbb{R}}=\bm{1}_A+\bm{1}_B=\bm{1}_{A+x}+\bm{1}_{B+y} \bmod{\mathscr{M}}
\end{split}
\end{displaymath}
for each $y \in \mathbb{R}$ with $(x,y) \in \mathcal{D}(A,B)$. Hence, $(x,x) \in \mathcal{D}(A,B)$ and, in addition, $(x,y) \in \mathcal{D}(A,B)$ if and only if $y-x \in \mathcal{D}(B)$ if and only if $y-x=rz$ for some $z \in \mathbb{Z}$. 

\medskip

 \ref{item2twodimensional} $\implies$  \ref{item2Btwodimensional} $\implies$ \ref{item3twodimensional}. They are obvious. 

 \medskip

  \ref{item3twodimensional} $\implies$ \ref{item1twodimensional}. 
By hypothesis there exists an injective sequence $(x_n)$ with values in $\mathcal{D}(A,B)$ which is convergent to $0$, with $x_n:=(x_{n,1}, x_{n,2})$ for all $n \in \omega$. In the case that there exist $i \in \{1,2\}$ and infinitely many $n\in \omega$ such that $x_{n,i}=0$, we would conclude by Theorem \ref{thm:maincategorycaseone} that at least one among $A,A^c,B,B^c$ would be meager. Hence we can suppose without loss of generality that $x_{n,i}\neq 0$ for all $i \in \{1,2\}$ and $n \in \omega$. Applying Lemma \ref{lem:twodimensionalcase} to $(A_1,A_2)=(A,B)$, we conclude that $A\cap B\sim \emptyset$. In addition, the identity $\bm{1}_{A+x_{n,1}}+\bm{1}_{B+x_{n,2}}
=\bm{1}_A+\bm{1}_B \bmod{\mathscr{M}}$ for each $n \in \omega$ 
implies that 
$$
(A^c+x_{n,1}) \cap (B^c+x_{n,2})
=A^c \cap B^c \bmod{\mathscr{M}}. 
$$
Hence, applying Lemma \ref{lem:twodimensionalcase} to $(A_1,A_2)=(A^c,B^c)$, we conclude that $A^c\cap B^c\sim \emptyset$. Therefore $\{A,B\}$ is a partition of $\mathbb{R}$ modulo $\mathscr{M}$. 
\end{proof}

\medskip

\begin{proof}[Proof of Theorem \ref{thm:generalcase}] We proceed similarly as in the proof of Theorem \ref{thm:twodimensional}. 

\medskip

\ref{item:1general} $\implies$ \ref{item:2general}. Fix 
the integer $h$ as in item \ref{item:1general}. Then 
\begin{displaymath}
    \begin{split}
\sum_{i=1}^k \bm{1}_{A_i}
&=\sum_{i=1}^k \sum_{\substack{I\subseteq \{1,\ldots,k\}: \\ i \in I,\, |I|=h} } \bm{1}_{\bigcap_{j \in I}A_j}\\
&=h \sum_{\substack{I\subseteq \{1,\ldots,k\}: \\ |I|=h} } \bm{1}_{\bigcap_{j \in I}A_j}=h \,\bmod{\mathscr{M}}.
\end{split}
\end{displaymath}


\medskip

\ref{item:2general} $\implies$ \ref{item:3general}. Fix $x_1\in \mathbb{R}$. Then 
$$
\sum_{i=1}^k \bm{1}_{A_i+x_1}=h=\sum_{i=1}^k \bm{1}_{A_i}=\sum_{i=1}^k \bm{1}_{A_i+x_i}\,\,\bmod{\mathscr{M}}
$$
for all $(x_2,\ldots,x_k) \in \mathbb{R}^{k-1}$ such that $(x_1,\ldots,x_k) \in \mathcal{D}(A_1,\ldots,A_k)$. It follows that $\sum_{i=2}^k (\bm{1}_{A_i+x_1}-\bm{1}_{A_i+x_i})=0$ modulo meager sets, hence
    $$
    \sum_{i=2}^k \bm{1}_{A_i}=\sum_{i=2}^k\bm{1}_{A_i+(x_i-x_1)}\,\,\bmod{\mathscr{M}}.
    $$
Therefore $(x_2-x_1,\ldots,x_k-x_1) \in \mathcal{D}(A_2,\ldots,A_k)$. 

\medskip

\ref{item:3general} $\implies$ \ref{item:5general} $\implies$ \ref{item:6general} $\implies$ \ref{item:4general}. They are obvious. 

\medskip

\ref{item:4general} $\implies$ \ref{item:1general}. 
For each $n\in \omega$, choose 
$x_n=(x_{n,1},\ldots,x_{n,k})\in \mathcal{D}(A_1,\ldots,A_k)$ such that 
$\|x_n\|<1/(n+1)$ and $x_{n,i}\neq 0$ for all $i\in \{1,\ldots,k\}$. Passing to a subsequence if necessary, we may suppose that $(x_n)$ is injective. For each $S\subseteq \mathbb{R}$, let us write $S^1:=S$ and $S^0:=S^c$. For each $j \in \{0,1,\ldots,k\}$, define 
$$
\mathscr{A}_j:=\left\{A_1^{\varepsilon_1} \cap \cdots \cap A_k^{\varepsilon_k}: \varepsilon \in \{0,1\}^k \text{ and }\sum_{i=1}^k \varepsilon_i=j\right\},
$$
and set 
$$
U_j:=\bigcup \mathscr{A}_j.
$$

We claim that each $U_j$ is either meager or comeager. Fix $j \in \{0,1,\ldots,k\}$ and enumerate
$$
\mathscr{A}_j=\{B_1,\ldots,B_q\},
$$
where
$$
B_t=A_1^{\varepsilon^t_1}\cap \cdots \cap A_k^{\varepsilon^t_k}
\quad\text{and}\quad
\sum_{i=1}^k \varepsilon_i^t=j
$$
for each $t\in \{1,\ldots,q\}$. For each pair $(t,i)$, put
$$
C_{t,i}:=A_i^{\varepsilon_i^t}.
$$
After reindexing the family $(C_{t,i}:1\le t\le q,\,1\le i\le k)$ as $C_1,\ldots,C_p$, where $p:=kq$, let $P_t$ be the set of indices corresponding to the block $(C_{t,1},\ldots,C_{t,k})$. Moreover, define
$$
y_{n,t,i}:=x_{n,i}
$$
for all $n\in \omega$, $t\in \{1,\ldots,q\}$, and $i\in \{1,\ldots,k\}$.

Since $x_n\in \mathcal{D}(A_1,\ldots,A_k)$, we have
$$
\sum_{i=1}^k \bm{1}_{A_i+x_{n,i}}
=
\sum_{i=1}^k \bm{1}_{A_i}
\quad \bmod{\mathscr{M}}.
$$
Therefore
\begin{displaymath}
\begin{split}
\bigcup_{t=1}^q \bigcap_{i=1}^k(C_{t,i}+y_{n,t,i})
&=
\bigcup_{t=1}^q \bigcap_{i=1}^k(A_i^{\varepsilon_i^t}+x_{n,i})\\
&=
\left\{z\in \mathbb{R}: \sum_{i=1}^k \bm{1}_{A_i+x_{n,i}}(z)=j\right\}\\
&\sim
\left\{z\in \mathbb{R}: \sum_{i=1}^k \bm{1}_{A_i}(z)=j\right\}\\
&=
\bigcup_{t=1}^q \bigcap_{i=1}^k A_i^{\varepsilon_i^t}
=
U_j.
\end{split}
\end{displaymath}
Each set $C_{t,i}$ has the Baire property, is neither meager nor comeager, and satisfies $(\partial C_{t,i})^\prime=\emptyset$. Hence Proposition \ref{prop:generaldimensionalcasenew} applies and gives that $U_j$ is either meager or comeager.

Since $U_0,\ldots,U_k$ are pairwise disjoint and cover $\mathbb{R}$, there exists exactly one $h\in \{0,\ldots,k\}$ such that $U_h$ is comeager, while $U_j$ is meager for all $j\neq h$. Moreover, $h\notin \{0,k\}$: if $h=0$, then $A_i^c$ is comeager for each $i$, and if $h=k$, then $A_i$ is comeager for each $i$, contradicting the standing assumption that each $A_i$ is neither meager nor comeager. Thus $h\in \{1,\ldots,k-1\}$.

Finally, define
$$
\mathscr{B}:=\left\{\bigcap_{i \in I} A_i: I\subseteq \{1,\ldots,k\} \text{ and }|I|=h\right\}. 
$$
We show that $\mathscr{B}$ is a partition of $\mathbb{R}$ modulo $\mathscr{M}$. First, since $U_h$ is comeager and $U_h\subseteq \bigcup \mathscr{B}$, the union $\bigcup \mathscr{B}$ is comeager. Second, if $I,J\subseteq \{1,\ldots,k\}$ are distinct sets of cardinality $h$, then
$$
\left(\bigcap_{i\in I}A_i\right)\cap \left(\bigcap_{j\in J}A_j\right)
\subseteq \bigcup_{\ell=h+1}^k U_\ell,
$$
which is meager. Therefore $\mathscr{B}$ is a partition of $\mathbb{R}$ modulo $\mathscr{M}$, as required.

\medskip

The last sentence follows by the equivalence of items \ref{item:4general}, \ref{item:5general}, and \ref{item:6general}.
\end{proof}

\medskip

\begin{proof}
    [Proof of Proposition \ref{prop:equivalentE1E5}] 
    The proof of the implications 
    $\ref{item:1general}$ 
    $\implies \ref{item:2general}$ 
    $\implies \ref{item:3general}$ 
    $\implies \ref{item:5general}$ 
    goes exactly as in Theorem \ref{thm:generalcase}. 

    \medskip

    $\ref{item:5general}\implies \ref{item:2general}$.  
Define the map $f: \mathbb{R} \to \{0,1,\ldots,k\}$ by $$
\forall x \in \mathbb{R}, \quad f(x):=\sum_{i=1}^k \bm{1}_{A_i}(x),
$$
and set $E_m:=f^{-1}(\{m\})$ for each $m \in \{0,1,\ldots,k\}$. Of course, each $E_m$ has the Baire property. Since $\Delta_k \subseteq \mathcal{D}(A_1,\ldots,A_k)$, we get $\sum_{i=1}^k(\bm{1}_{A_i+y}-\bm{1}_{A_i})=0 \ \bmod{\mathscr{M}}$, that is, $f(x-y)=f(x)\bmod{\mathscr{M}}$ for all $x,y \in \mathbb{R}$. 
For every $m \in \{0,1,\ldots,k\}$ and $y \in \mathbb{R}$ we have
\[
(E_m+y)\bigtriangleup E_m
\subseteq 
\{x\in\mathbb{R}: f(x-y)\neq f(x)\}\in\mathscr{M},
\]
so that $E_m+y\sim E_m$. It follows that $\mathcal{D}(E_m)=\mathbb{R}$ for every $m$. By Theorem \ref{thm:maincategorycaseone}, each $E_m$ is either meager or comeager. Since $\{E_0,E_1,\ldots,E_k\}$ is a partition of $\mathbb{R}$, exactly one of them is comeager. Thus there exists $h\in\{0,1,\ldots,k\}$ such that 
$$
f=h \ \bmod{\mathscr{M}}.
$$ 
Because each $A_i$ is neither meager nor comeager, we cannot have $h\in\{0,k\}$ (otherwise all $A_i$ would be meager or all comeager). Therefore $h\in\{1,2,\ldots,k-1\}$.

\medskip

    $\ref{item:2general}\implies \ref{item:1general}$. 
    Pick $h\in\{1,2,\ldots,k-1\}$ such that $f=h \ \bmod{\mathscr{M}}$. 
    Define
\[
\mathscr{B}:=\left\{\,\bigcap_{i\in I}A_i:\ I\subseteq\{1,\ldots,k\}\text{ and }|I|=h\right\}.
\]
Fix $x\in E_h$ and set $I_x:=\{i\in\{1,\ldots,k\}: x\in A_i\}$. Then $|I_x|=h$ and $x\in\bigcap_{i\in I_x}A_i$. Hence $E_h\subseteq \bigcup \mathscr{B}$, so $\bigcup \mathscr{B}$ is comeager.

Now take distinct $I,J\subseteq\{1,\ldots,k\}$ with $|I|=|J|=h$. If $x\in \bigcap_{t\in I\cup J}A_t$, then $f(x)\ge |I\cup J|\ge h+1$. It follows that 
\[
\bigcap_{t\in I\cup J}A_t
\subseteq 
\{x\in\mathbb{R}: f(x)\neq h\}
\in\mathscr{M}.
\]
Thus the elements of $\mathscr{B}$ are pairwise disjoint modulo $\mathscr{M}$ and their union is comeager, i.e.,  $\mathscr{B}$ is a partition of $\mathbb{R}$ modulo $\mathscr{M}$. This completes the proof. 
\end{proof}

\bibliographystyle{amsplain}
\bibliography{ideale}

\end{document}